\documentclass{amsart}

\usepackage[spanish,USenglish]{babel} 
\usepackage[utf8]{inputenc}
\usepackage{float}
\usepackage{hyperref}
\usepackage{amssymb}
\usepackage{amscd}
\usepackage{amsmath}
\usepackage{amsthm}
\usepackage{array}
\usepackage{tikz}
\usepackage{hyperref}

\def\aut{\operatorname{Aut}}
\def\Sym{\operatorname{Sym}}
\def\exp{\operatorname{exp}}
\def\push{\operatorname{push}}
\def\spref{\operatorname{spref}}

\newtheorem{theorem}{Theorem}[section]
\newtheorem{proposition}[theorem]{Proposition}
\newtheorem{corollary}[theorem]{Corollary}
\newtheorem{lemma}[theorem]{Lemma}
\theoremstyle{definition}
\newtheorem{question}[theorem]{Question}
\newtheorem{remark}[theorem]{Remark}
\newtheorem{observation}[theorem]{Observation}
\newtheorem{example}[theorem]{Example}
\newtheorem{definition}[theorem]{Definition}
\newcommand{\N}{\mathbb{N}}
\newcommand{\Cn}{\mathfrak{C}_n}
\newcommand{\Cm}{\mathfrak{C}_m}

\newcommand\blfootnote[1]{%
  \begingroup
  \renewcommand\thefootnote{}\footnote{#1}%
  \addtocounter{footnote}{-1}%
  \endgroup
}

\newcommand{\concat}{\vert\!\vert}
\newcommand{\ew}{\varepsilon}
\newcommand{\seteq}{:=}
\newcommand{\vmg}{V_m(G)}
\newcommand{\vnh}{V_n(H)}

\begin{document}

\title[Embeddings between symmetric Thompson groups]{Some embeddings between symmetric R. Thompson groups}

\author{Julio Aroca}
\address{Instituto de Ciencias Matem\'aticas, Madrid, Spain.}
\email{julio.aroca@icmat.es}

\author{Collin Bleak}
\address{University of St Andrews, St Andrews, Scotland.}
\email{cb211@st-andrews.ac.uk}

\blfootnote{The first author wants to acknowledge financial support from the Spanish Ministry of Economy and Competitiveness, through the “Severo Ochoa Program for Centres of Excellence in R\&D” (SEV-2015-0554). The second author wishes to  acknowledge support from EPSRC grant EP/R032866/1 during the creation of this paper.
}

\date{\today}

\begin{abstract}
Let $m\leq n\in \N$, and $G\leq \Sym(m)$ and $H\leq \Sym(n)$.  In this article we find conditions enabling embeddings between the symmetric R. Thompson groups $\vmg$ and $\vnh$. When $n\equiv 1 \mod(m-1)$, and under some other technical conditions, we find an embedding of $\vnh$ into $\vmg$ via topological conjugation. With the same modular condition we also generalise a purely algebraic construction of Birget from 2019 to find a group $H\leq \Sym(n)$ and an embedding of  $\vmg$ into $\vnh$. 
\end{abstract}

\maketitle



\section{Introduction}

This article concerns embeddability conditions for pairs of groups from the family of \emph{symmetric R. Thompson groups} 
$\{\vmg\}$.  The group $\vmg$ is the group $\vmg=\langle V_m\cup
G\rangle$, where $V_m\leq \aut(\Cm)$ is the Higman-Thompson group
denoted $G_{m,1}$ by Higman in \cite{HI}, acting on the Cantor space $\Cm\seteq \{0,1,\ldots, m-1\}^\omega$, while $G$ is a particular faithful representation of a finite group $\widetilde{G}\leq \Sym(m)$ in $\aut(\Cm)$.

The groups $\{\vmg\}$ have developed as groups of interest for a variety of reasons.  Firstly, they were singled out as natural groups of interest in \cite{NekraCuntz} and \cite{Roever}, and they arise naturally as a fundamental subfamily of Hughes' $\mathcal{F}\mathcal{S}\mathcal{S}$ groups \cite{Hughes}.  The paper \cite{BDJ} shows that for $n\geq2$, $V_m\cong \vmg$ if and only if $\widetilde{G}$ is semiregular (the nontrivial elements of $\widetilde{G}$ have no fixed points), and also, that for $m>3$ there exists $\widetilde{G},\widetilde{H}\in\Sym(m)$ with $\widetilde{G}\cong\widetilde{H}$ but where the induced groups $\vmg$ and $V_m(H)$ are not isomorphic (the orbit structure of the actions of the elements of the groups $\widetilde{G}$ and  $\widetilde{H}$ impacts the isomorphism types of the groups $\vmg$ and $V_m(H)$).  In another direction, in \cite{FarleyCoCF} Farley shows the symmetric R. Thompson groups are CoCF groups (see \cite{HRRT} for the definition of CoCF groups).  Thus, if one can show that some group in the family $\{\vmg\}$ fails to embed in $V=V_2$, then Lehnert's conjecture will be shown to be false (see \cite{LehnertDiss,LehnertSchweitzer, BMN}).

We investigate conditions on $m\leq n$, $G\leq \Sym(m)$, and $H\leq \Sym(n)$ that guarantee the existence of embeddings between the groups $\vmg$ and $\vnh$ (we now drop the ``tilde'' notation on the groups $G$ and $H$ when thinking of them as subgroups of $\Sym(m)$ and $\Sym(n)$, respectively).  Thus, this note can be thought of as a continuation of the investigations in \cite{BDJ}, and is partly inspired by the work of Birget (in \cite{B}, he gives a method to embed $V_2$ into $V_m$ for $2\leq m$ (embeddings in the other direction have been known since Higman's book \cite{HI})), and partly by considering some of the questions alluded to in the previous paragraph.  In this context, our two embedding results depend on the direction of the embedding ($\vmg\rightarrowtail \vnh$ or $\vnh\rightarrowtail \vmg$), and our constructed embeddings require in both cases the Higman condition $n\equiv 1\mod(m-1)$.

The embedding $\vmg\rightarrowtail \vnh$ is algebraic in nature, inspired by the embedding of Birget from $V_2$ into $V_n$ in \cite{B}, while the embedding $\vnh\rightarrowtail \vmg$ uses a topological conjugacy by rational group elements (see \cite{GNS}).

We can now state and discuss our main results.

\begin{theorem}\label{thm:MT1}
Let $n,m \geq 2$ be natural numbers such that $m<n$, and let $G \leq \Sym(m)$, $H \leq \Sym(n)$. Suppose that:
\begin{enumerate}
\item There exists an prefix code $\widetilde{A}$ of $\Cm$ such that $\vert \widetilde{A} \vert = n$,
\item the group $\mathcal{R}_{G}(\widetilde{A})$ is well defined, and 
\item $\mathcal{R}_{G}(\widetilde{A})$ and $H$ are cyclically isomorphic.
\end{enumerate} Then $\vnh$ embeds in $\vmg$.
\end{theorem}
For this first result, observe that we require a group $\mathcal{R}(G)$ to be well defined.  This group (when it exists) can be thought of in a natural way as a subgroup of $\Sym(n)$.  We also need $\mathcal{R}(G)$ to be cyclically isomorphic to $H$.  This is equivalent to the existence of an $n$-element complete prefix code $\widetilde{A}$ in $\mathcal{A}_m^*$ which is preserved by the action of the iterated permutations of $G$ on $\mathcal{A}_m^*$ and where the cycle structure of the action of $G$ on $\widetilde{A}$ is isomorphic (element-for-element) to the cycle structure of the action of $H$ on $\mathcal{A}_{n}$.

Thus, we have the following general observation.
\begin{observation}
In practical terms, natural applications of Theorem \ref{thm:MT1} occur by choosing $m$ and a prefix code that is closed under the action of iterated  permutations from some $G\leq \Sym(m)$.  Then, one immediately obtains a cyclically isomorphic group $H$ in $\Sym(n)$.
\end{observation}

\begin{theorem}\label{thm:MT2}
Let $n,m \geq 2$ be natural numbers such that $n = k(m-1)+ m$ for some $k \geq 1$, and let $G \leq \Sym(m)$. Let $H = G_{ext} \leq \Sym(n)$ the extended symmetric group of $H$, whose elements act as the elements of $H$ on the first $m$ elements, and act as the identity on the remaining $n-m$.

Then $\vmg$ embeds in $\vnh$.
\end{theorem}
Thus, to achieve such an embedding, the main obstruction is that $m$ and $n$ satisfy the Higman Condition, i.e.:
\begin{corollary} Let $m\leq n$ so that $n\equiv m\mod (m-1)$, and let $G\leq \Sym(m)$.  Then, there exists $H\leq \Sym(n)$ so that $\vmg\rightarrowtail \vnh$.
\end{corollary}

As mentioned above, for given $m$ and $G$ semiregular, the paper \cite{BDJ} shows $\vmg\cong V_m$, while Higman's book \cite{HI} gives an embedding of $V_m$ into R. Thompson's group $V=V_2$.  The semiregularity condition above has to do with local groups of germs of the action of $\vmg$ on $\Cm$ (see \cite{BDJ,BL}).  Our topological embedding preserves the local groups of germs, while the algebraic embeddings (for $G$ non-trivial) produce complicated local groups of germs.  Thus, it is impossible to chain our families of embeddings together to get an embedding from $\vmg$ into $V_2$ when $G$ is not semi-regular.

We therefore ask the following question:
\begin{question}
Does there exist $m\in \N$ and $G\in \Sym(m)$ so that $G$ is not semiregular, but where there is an embedding from $\vmg$ into $V_2$?
\end{question}

{\flushleft{\it Acknowledgements:}}\\
The authors would like to thank Javier Aramayona, Jim Belk, and Matthew G. Brin for numerous and enjoyable conversations around the material in this note.


\section{Symmetric Thompson's groups}

In this section, we introduce symmetric Thompson's groups, giving an easy way to express its elements as tables (this corresponds to a generalised construction of Higman used by Scott and R\"over in the creation of their extensions of $V$ (\cite{ScottI,ScottII,ScottIII,Roever}).   

\subsection{Tables} For natural $n>1$, let $\Cn$ be the $n$-adic Cantor set, which is constructed inductively as follows: $\Cn^1$ corresponds to first subdividing $\Cn^0= [0,1]$ into $2n-1$ closed intervals of equal length (so, sharing endpoints with neighbours), numbered $1, \ldots, 2n-1$ from left to right, and then taking the collection of odd-numbered sub-intervals. Next, $\Cn^2$ is obtained from $\Cn^1$ by applying the same procedure to each of the intervals forming $\Cn^1$, and so on.  Then, $C_n$ is the limit of this process, so that 
\[
\Cn=\cap_i \Cn^i.
\]  Now, let $\mathcal{A}_n = \{0,\dots,n-1\}$ and give it the discrete topology. It is easy to build a direct homemorphism from the space $\mathcal{A}_n^\N$ equipped with the product topology to $\Cn$, so every element $\zeta \in \Cn$ can be expressed as an infinite word $\zeta = w_1w_2\dots $, where $w_i \in \mathcal{A}_n$.  It is a classical result of Brouwer from \cite{BRO} that all of the spaces in the set $\{\Cn\}$ are abstractly homeomorphic to each other.

We denote by $\mathcal{A}_n^*$ the set of finite words in $\mathcal{A}_n$. The empty word $\varepsilon$ is also in $\mathcal{A}_n^*$.

\begin{definition}[Concatenation]
Let $u = u_1u_2\dots u_k, u_i \in \mathcal{A}_n$ be a finite word and $v \in \mathcal{A}_n^* \cup \mathcal{A}_n^{\N}$ with $v=v_1v_2\ldots$ (where for all valid indices $i$ we have  $v_i\in \mathcal{A}_n$) . The \textit{concatenation} of $u$ with $v$ is the (finite or infinite) word: $$u \concat v = u_1u_2\dots u_kv_1v_2\ldots.$$ 
\end{definition}

With concatenation being a fundamental operation, we will often just write the concatenation of two strings without the formal concatenation operator, that is, we might write $u\concat v$ as simply $uv$, reserving the formal use of ``$\concat$'' for situations where we wish to stress that a concatenation is occuring.

\begin{definition}[Prefix order]
Let $u \in \mathcal{A}_n^*$  and $v \in \mathcal{A}_n^{*}\cup \mathcal{A}_n^{\N}$.  We say that $u$ is a \textit{prefix} of $v$ ($u \leq_{pref} v$) if $v = u\concat w$, for some $w \in\mathcal{A}_n^{*}\cup \mathcal{A}_n^{\N}$.
\end{definition}

Note that this property is transitive for finite length words: If $u \leq_{pref} v$ and $v \leq_{pref} w$ then $u \leq_{pref} w$. In addition, $u \leq_{pref} u$, as $\varepsilon \in \mathcal{A}_n^*$.  That is, $\leq_{pref}$ provides a partial order on $\mathcal{A}_n^{*}$.

\begin{definition}[Prefix code]
Let $S$ be a finite set of words in $\mathcal{A}_n^*$. Then $S$ is an \textit{prefix code} of $\Cn$ if for every infinite word $\zeta \in \mathcal{A}_n^{\N}$ there exists one and only one word $s \in S$ such that $s \leq_{pref} \zeta$. (Specifically, a prefix code is a complete anti-chain for the partial order $\leq_{pref}$.)
\end{definition}

For convenience, we will use the following notation: let $\sigma \in \Sym(n)$ be an element of the symmetric group of $n$ elements. Given any word $\zeta = z_1z_2z_3 \dots \in \mathcal{A}_n^\mathbb{*} \cup \mathcal{A}_n^\N$ we define $\sigma(\zeta) = \sigma(z_1)\sigma(z_2)\sigma(z_3)\dots \in \mathcal{A}_n^\mathbb{*} \cup \mathcal{A}_n^\N$. Let $\sigma_i \in H \leq \Sym(n)$.   (Note that we are using left actions here, so if $\sigma,\tau\in \Sym(n)$ then the product $\tau\sigma$ means employ the permutation $\sigma$ first, and then employ $\tau$).

With the above notation, an element of $\vnh$ is a homeomorphism of $\Cn$  that can be (non-uniquely) described by a \textit{table} as follows:
$$v = \begin{bmatrix}
p_1 & p_2 & \cdots & p_k \\
\sigma_1 & \sigma_2 & \cdots & \sigma_k\\
q_1 & q_2 & \cdots & q_k \\
\tau_1 & \tau_2 & \cdots & \tau_k
\end{bmatrix},$$
where $p_i,q_i \in \mathcal{A}_n^*$, $\sigma_i, \tau_i \in H$ and such that the sets $P = \{p_i\}_{i=1}^k$ and $Q = \{q_i\}_{i=1}^k$ are prefix codes of $\Cn$. We say that $k \geq 1$ is the \textit{length} of the table. The homeomorphism of $\Cn$ induced can be defined as follows: for every infinite word $\zeta$ such that $p_i\leq_{pref}\zeta$, that is $\zeta = p_i \concat u$ for some $u \in \mathcal{A}_n^\N$, we have 
$$v: p_i \concat \sigma_i(u) \rightarrow q_i \concat \tau_i(u).$$
There are infinitely many tables which induce the same homeomorphism of $\Cn$. We proceed to define the four basic moves we can perform on a table in order to obtain an equivalent one (the four basic moves naturally split as two essential sorts of moves, together with their inverse (or ``near-inverse'') moves). 

The first basic move is \textit{expansion}: for a given prefix code $$P = \{p_1, \dots, p_i , \dots , p_k\},$$ we can consider $$\widetilde{P} = \{p_1, \dots, p_i0, \dots, p_i(n-1) , \dots , p_k\}$$ by expanding the word $p_i$. This expansion not only occurs in $P$, as the image of $p_i$ must be also expanded. So we have $$\widetilde{Q} = \{q_1, \dots, q_i0, \dots, q_i(n-1) , \dots  ,q_k\}.$$ It is easy to see that both $\widetilde{P}$ and $\widetilde{Q}$ are also prefix codes. Then:
$$\begin{bmatrix}
p_1 & \cdots& p_i & \cdots & p_k \\
\sigma_1 & \cdots& \sigma_i & \cdots & \sigma_k \\
q_1 & \cdots& q_i & \cdots & q_k \\
\tau_1 & \cdots& \tau_i & \cdots & \tau_k
\end{bmatrix} \equiv 
\begin{bmatrix}
p_1 & \cdots& p_i \sigma_i(0) & \cdots & p_i \sigma_i(n-1) & \cdots & p_k \\
\sigma_1 & \cdots& \sigma_i & \cdots& \sigma_i & \cdots & \sigma_k \\
q_1 & \cdots& q_i \tau_i(0) & \cdots & q_i \tau_i(n-1) & \cdots & q_k \\
\tau_1 & \cdots& \tau_i & \cdots& \tau_i & \cdots & \tau_k \\
\end{bmatrix}.$$  One can always perform an expansion, but not all tables look like the result of an expansion.  Naturally, the inverse of an expansion (when it is defined) is called a \emph{reduction}.

The second move we can perform on a table is \textit{pushing down} (resp. \textit{pushing up}) the action of all $\sigma_i$ such that $\sigma_i = Id$ for every $i \in \{1, \dots, k\}$ (resp. $\tau_i = Id$ for every $i \in \{1, \dots, k\}$):
\begin{align*}
\begin{bmatrix}
p_1 & p_2 & \cdots & p_k \\
\sigma_1 & \sigma_2 & \cdots & \sigma_k \\
q_1 & q_2 & \cdots & q_k \\
\tau_1 & \tau_2 & \cdots & \tau_k 
\end{bmatrix} &\equiv 
\begin{bmatrix}
p_1 & p_2 & \cdots & p_k \\
Id & Id & \cdots & Id\\
q_1 & q_2 & \cdots & q_k \\[2pt]
\tau_1\sigma_1^{-1} & \tau_2\sigma_2^{-1} & \cdots & \tau_k\sigma_k^{-1}
\end{bmatrix} \\ &\equiv 
\begin{bmatrix}
p_1 & p_2 & \cdots & p_k \\[2pt]
\sigma_1\tau_1^{-1} & \sigma_2\tau_2^{-1} & \cdots & \sigma_k\tau_k^{-1} \\[2pt]
q_1 & q_2 & \cdots & q_k \\
Id & Id & \cdots & Id\\
\end{bmatrix}.
\end{align*}

It is not hard to see that a table gives a well-defined homeomorphism of the appropriate Cantor space, and if two tables are related by a finite sequence of our four moves then they represent the same homeomorphism.  The reader can also check that if a homeomorphism of an appropriate Cantor space is represented by two tables, then in fact these tables are in the same equivalence class under our four basic moves on tables.  Thus, we can just consider our group elements to be the equivalence classes of tables with the aforementioned relations. 

The {composition} of two different elements $u,v\in \vnh$ is easy to compute using the equivalences. Let $u$, $v \in \vnh$, such that $u$ takes the prefix code $P$ to the prefix code $Q$ (resp. $v$ takes $P'$ to $Q'$). We need to find a prefix code $S$ such that, for every element $s \in S$, there exists one element $q \in Q$ and one element $p' \in P'$ such that $q\leq_{pref} s$ and $p'\leq_{pref} s$. This can always be done by expanding $P'$ and $Q$ until we obtain the same prefix code $S$. Thus, without loss of generality:
$$u = \begin{bmatrix}
p_1 & p_2 & \cdots & p_k \\
\sigma_1 & \sigma_2 & \cdots & \sigma_k \\
s_1 & s_2 & \cdots & s_k \\
\tau_1 & \tau_2 & \cdots & \tau_k
\end{bmatrix}, \quad v = \begin{bmatrix}
s_1 & s_2 & \cdots & s_k \\
\sigma'_1 & \sigma'_2 & \cdots & \sigma'_k \\
q'_1 & q'_2 & \cdots & q'_k \\
\tau'_1 & \tau'_2 & \cdots & \tau'_k
\end{bmatrix}.$$

Finally, we push up the action of $u$ and push down the action of $v$:
\begin{align*}
u&  = \begin{bmatrix}
p_1 & p_2 & \cdots & p_k \\
\sigma_1\tau_1^{-1} & \sigma_2\tau_2^{-1} & \cdots & \sigma_k\tau_k^{-1}\\
s_1 & s_2 & \cdots & s_k \\
Id & Id & \cdots & Id
\end{bmatrix},\\ v & = \begin{bmatrix}
s_1 & s_2 & \cdots & s_k \\
Id & Id & \cdots & Id\\
q'_1 & q'_2 & \cdots & q'_k \\[2pt]
\tau'_1(\sigma'_1)^{-1} & \tau'_2(\sigma'_2)^{-1} & \cdots & \tau'_k(\sigma'_k)^{-1}
\end{bmatrix},
\end{align*}
so 
$$v \circ u  = \begin{bmatrix}
p_1 & p_2 & \cdots & p_k \\
\sigma_1\tau_1^{-1}& \sigma_2\tau_2^{-1} & \cdots & \sigma_k\tau_k^{-1}\\
q'_1 & q'_2 & \cdots & q'_k \\[2pt]
\tau'_1(\sigma'_1)^{-1} & \tau'_2(\sigma'_2)^{-1} & \cdots & \tau'_k(\sigma'_k)^{-1}
\end{bmatrix}.$$

We sum up the previous discussion in the following proposition:
\begin{proposition}
$\vnh$ is a group with the composition.
\end{proposition}


\section{Topological Embeddings}

In this section, we present topological embeddings between symmetric Thompson's groups. The key idea is, given any group $\vnh$, to translate the action of an element $\sigma \in H$ into a permutation $\widetilde{\sigma}$ of the elements of some prefix code of $\Cm$. Therefore, $\widetilde{\sigma} \in \vmg$ for some $G$.

Our method will be first to understand when actions on prefix codes over smaller alphabets can represent embeddings of permutations on larger alphabets which commute with our core operations of expansion and contraction of prefix codes.  With that understanding in hand, we can then build the desired embedding from a group $\vnh$ to a group $\vmg$ for $m\leq n$.

We first establish some useful definitions.

\subsection{The Root Group \texorpdfstring{$\boldsymbol{\mathcal{R}_{G}(S)}$}{RGS}}

Given a linear order $\leq$ on $\mathcal{A}_n$ (we choose $0<1<\ldots<n-1$), there is an induced standard dictionary order:
\begin{definition}[Dictionary order]
Let $\mathcal{A}_n = \{a_0, \dots, a_{n-1}\}$ be an alphabet with linear order $\leq$. We define the \textit{dictionary order} on $\mathcal{A}_n^*$ as follows. Let $u,v \in \mathcal{A}_n^*$, then $u \leq_{dict} v$ if and only if:
\begin{enumerate}
\item $u \leq_{pref} v$,
\item $u \not\leq_{pref} v$ and there exist $p, s, t \in \mathcal{A}_n^*$ and $\alpha, \beta \in \mathcal{A}_n$ such that $u = p\alpha s, v = p \beta t$, and $\alpha < \beta$.
\end{enumerate} 
\end{definition}

Let $2\leq m <n\in\N$ be fixed, and let $G \leq \Sym(m)$. We assume for the construction Higman's condition: $n = k(m-1)+ 1$ for some $k \geq 0$ (so, $n\equiv 1\mod (m-1)$).  Then, we can find $S = \{s_0, \dots, s_{n-1}\}$ an ordered prefix code of $\Cm$ of length $n$ (using the dictionary order).   Observe for now that if $\sigma(S) = S \ \forall \sigma \in G$, then $\sigma: S \rightarrow S$ will induce the desired rearrangement $\widetilde{\sigma}$ of the elements of $S$, that is, $\widetilde{\sigma} \in \Sym_S\cong \Sym(n)$. 

We will be interested in the set $\mathcal{T}$ of triples $(m,n,G)$
where $m\leq n$, $n\equiv 1\mod(m-1)$, and $G\in \Sym(m)$.  We will say a triple $(m,n,G)\in \mathcal{T}$ is \emph{satisfiable} if there is a prefix code $S\subset X_m^*$ with $|S|=n$ and where for each $\sigma\in G$ the action of $\sigma$ on $X_m^*$ preserves $S$.  In this case we say \emph{$S$ is a solution for the triple $(m,n,G)$.}

\begin{definition}[Root group]
Given a triple $(m,n,G)\in\mathcal{T}$ and a solution $S$, 
we define the \textit{root group} $\mathcal{R}_{G}(S)$ to be the group of permutations of $S$ induced by the action of $G$ on $S$.
\end{definition}

It is the case that not every triple $(m,n,G)\in\mathcal{T}$ admits a solution $S$.  However, when it does admit a solution $S$, it is immediate that $\mathcal{R}_{G}(S)$ is isomorphic to $G$.

\begin{definition}[Cycle type]
Let $\sigma \in \Sym(n)$. The \textit{cycle type c} of $\sigma$ is the multiset (a set, but allowing multiple elements that are equal) of lengths of the cycles in the cycle decomposition of $\sigma$. We say that two subgroups $H, H' \in \Sym(n)$ are \textit{cyclically isomorphic} if there exists an isomorphism $\psi: H \rightarrow H'$ which preserves the cycle type of every permutation, that is, $c(\sigma) = c(\psi(\sigma))$ for all $\sigma \in H$. 
\end{definition}

\begin{remark}
 Subgroups $H$ and $H'$ are cyclically isomorphic if and only if they are conjugate in $\Sym(n)$, but our focus is on cycle structure and that is why we are using the language we have chosen. (N.B., there exist exotic automorphisms of $\Sym(6)$ which do not arise by conjugation, but these automorphisms change the cycle structure of some elements of order two.) 
\end{remark}

\subsection{The induced group \texorpdfstring{$\boldsymbol{\mathcal{G}}$}{G}} 

We proceed to define the group $\mathcal{G} < \vmg$ such that $\vnh$ is isomorphic to $\mathcal{G}$. 

Let $\vmg$ and $\vnh$ be two symmetric Thompson's groups such that $m$ and $n$ fulfill Higman's condition. Let $\widetilde{\mathcal{A}}_m = \{a_0,a_1,\dots,a_{n-1}\}$ be an ordered prefix code of $\Cm$ of length $n$, such that $\mathcal{R}_{G}(\widetilde{\mathcal{A}}_m)$ is well defined. We may think of  $\mathcal{R}_{G}(\widetilde{\mathcal{A}}_m)$ as a subgroup of $\Sym(n)$ by using the bijection from $\widetilde{\mathcal{A}}_m$ to $\mathcal{A}_n$ induced by the lexicographic ordering of $\widetilde{\mathcal{A}}_m$.  Thus, we can define $\mathcal{G}$ as the set of equivalence classes of tables of the form:
$$v = \begin{bmatrix}
p_1 & p_2 & \cdots & p_k \\
\sigma_1 & \sigma_2 & \cdots & \sigma_k \\
q_1 & q_2 & \cdots & q_k \\
\tau_1 & \tau_2 & \cdots & \tau_k
\end{bmatrix},$$
where $\sigma_i,\tau_i \in \mathcal{R}_{G}(\widetilde{\mathcal{A}}_m) \ \forall i$ and the prefix codes $P$ and $Q$ consist of words in the alphabet $\widetilde{\mathcal{A}}_m$, that is, every $p_i$ and $q_i$ is a nonempty concatenation of elements of $\widetilde{\mathcal{A}}_m$. 

It is straightforward to check that $\mathcal{G}$ is a group, since the concatenation of two words in $\widetilde{\mathcal{A}}_m^*$ is another word in $\widetilde{\mathcal{A}}_m^*$. Tables of $\mathcal{G}$ are well defined by expansions and pushings, and the action of an element $\sigma \in \mathcal{R}_{G}(\widetilde{\mathcal{A}}_m)$ on a word in $\mathcal{\widetilde{\mathcal{A}}}_m$ gives another word in $\mathcal{\widetilde{\mathcal{A}}}_m$ by the definition of root group, so the composition of any two elements in $\mathcal{G}$ gives another element of $\mathcal{G}$.

\begin{proof}[Proof of Theorem \ref{thm:MT1}]
Let $\vmg$ and $\vnh$ two symmetric Thompson's groups. Let $\mathcal{A}_n = \{0,1,\dots,n-1\}$ and $\widetilde{\mathcal{A}}_m = \{a_0,a_1,\dots,a_{n-1}\}$ such that $\mathcal{R}_G(\widetilde{\mathcal{A}}_m)$ is well defined. We define the following \textit{translating map}:
$$ \begin{array}{cccc} \widetilde{t}: & \mathcal{A}_n &\longrightarrow & \widetilde{\mathcal{A}}_m \\
& i & \longrightarrow &  a_i,
\end{array}$$
and $$\begin{array}{cccc} t: & H &  \longrightarrow &  \mathcal{R}_{G}(\widetilde{\mathcal{A}}_m) \\
& \sigma & \longrightarrow & \widetilde{\sigma}\end{array},$$ 
where $t$ is the isomorphism between $H$ and $\mathcal{R}_{G}(\widetilde{\mathcal{A}}_m)$. Note that $t$ and $\widetilde{t}$ have the following property:
$$\widetilde{t}(\sigma(i))  
 = a_{\sigma(i)}
 = \widetilde{\sigma}(a_i)
 = t(\sigma)(a_i)
 = t(\sigma)(\widetilde{t}(i)) , \forall \sigma \in H, \forall i \in \mathcal{A}_n,$$
as $\widetilde{\sigma}(a_i) = a_{\sigma(i)}, \forall \sigma \in H, \forall i \in \mathcal{A}_n$, since $\mathcal{R}_{G}(\widetilde{\mathcal{A}}_m)$ and $H$ are cyclically isomorphic.

Our embedding is as follows:
$$v = \begin{bmatrix}
p_1 & p_2 & \cdots & p_k \\
\sigma_1 & \sigma_2 & \cdots & \sigma_k \\
q_1 & q_2 & \cdots & q_k \\
\tau_1 & \tau_2 & \cdots & \tau_k
\end{bmatrix} \overset{\iota}{\longrightarrow} \begin{bmatrix}
\widetilde{t}(p_1) & \widetilde{t}(p_2) & \cdots & \widetilde{t}(p_k) \\
t(\sigma_1) & t(\sigma_2) & \cdots & t(\sigma_k) \\
\widetilde{t}(q_1) & \widetilde{t}(q_2) & \cdots & \widetilde{t}(q_k) \\
t(\tau_1) & t(\tau_2) & \cdots &t(\tau_k)
\end{bmatrix}.$$
We see that $\iota$ commutes with expansions and pushings, that is, $\iota(\exp(v)) = \exp(\iota(v))$ and $\iota(\push(v)) = \push(\iota(v)) \ \forall v \in \vnh$:

\begin{align*}
\push(v) & = \begin{bmatrix}
p_1 & p_2 & \cdots & p_k \\
Id & Id & \cdots & Id \\
q_1 & q_2 & \cdots & q_k \\
\tau_1\sigma_1^{-1} & \tau_2\sigma_2^{-1} & \cdots & \tau_k\sigma_k^{-1}
\end{bmatrix}, \\ 
\iota(\push(v)) & = \begin{bmatrix}
\widetilde{t}(p_1) & \widetilde{t}(p_2) & \cdots & \widetilde{t}(p_k) \\
Id & Id & \cdots & Id \\
\widetilde{t}(q_1) & \widetilde{t}(q_2) & \cdots & \widetilde{t}(q_k) \\
t(\tau_1\sigma_1^{-1}) & t(\tau_2\sigma_2^{-1}) & \cdots & t(\tau_k\sigma_k^{-1})
\end{bmatrix}, \\
\push(\iota(v)) &  = \begin{bmatrix}
\widetilde{t}(p_1) & \widetilde{t}(p_2) & \cdots & \widetilde{t}(p_k) \\
Id & Id & \cdots & Id \\
\widetilde{t}(q_1) & \widetilde{t}(q_2) & \cdots & \widetilde{t}(q_k) \\
t(\tau_1)t(\sigma_1^{-1}) & t(\tau_2)t(\sigma_2^{-1}) & \cdots & t(\tau_k)t(\sigma_k^{-1})
\end{bmatrix}.
\end{align*}
As $t$ is an isomorphism of groups, the commutativity follows. On the other hand, suppose that we expand the prefix code $P = \{p_1, \dots, p_k\}$ on $p_i$ (we argue below that $\iota$ commutes with expanding, but our argument uses  the pushed version of $v$: it is easy to see that this is sufficient):
$$\exp(v) = \begin{bmatrix}
p_1  & \cdots & p_i0 & \cdots & p_i(n-1)  & \cdots &  p_k \\
Id  & \cdots & Id & \cdots & Id & \cdots & Id\\
q_1  & \cdots & q_i\concat\tau_i\sigma_i^{-1}(0) & \cdots & q_i\concat\tau_i\sigma_i^{-1}(n-1)  & \cdots & q_k \\
\tau_1\sigma_1^{-1} & \cdots & \tau_i\sigma_i^{-1} & \cdots & \tau_i\sigma_i^{-1} & \cdots & \tau_k\sigma_k^{-1}
\end{bmatrix}.$$
The table for $\iota(\exp(v))$ is: 
$$ \begin{bmatrix}
\widetilde{t}(p_1) & \cdots & \widetilde{t}(p_i0) & \cdots & \widetilde{t}(p_i(n-1))  & \cdots & \widetilde{t}(p_k) \\
Id  & \cdots & Id & \cdots & Id & \cdots & Id\\
\widetilde{t}(q_1)  & \cdots & \widetilde{t}(q_i\concat\tau_i\sigma_i^{-1}(0)) & \cdots & \widetilde{t}(q_i\concat\tau_i\sigma_i^{-1}(n-1))  & \cdots & \widetilde{t}(q_k) \\
t(\tau_1\sigma_1^{-1})  & \cdots & t(\tau_i\sigma_i^{-1}) & \cdots & t(\tau_i\sigma_i^{-1}) & \cdots & t(\tau_k\sigma_k^{-1})
\end{bmatrix},$$
Finally, the table for $\exp(\iota(v))$ is: 
$$\begin{bmatrix}
\widetilde{t}(p_1)  & \cdots & \widetilde{t}(p_i)a_0 & \cdots & \widetilde{t}(p_i)a_{n-1}  & \cdots & \widetilde{t}(p_k) \\
Id  & \cdots & Id & \cdots & Id & \cdots & Id\\
\widetilde{t}(q_1) & \cdots & \widetilde{t}(q_i)\concat t(\tau_i\sigma_i^{-1})(a_0) & \cdots &\widetilde{t}(q_i)\concat t(\tau_i\sigma_i^{-1})(a_{n-1})  & \cdots & \widetilde{t}(q_k) \\
t(\tau_1\sigma_1^{-1})  & \cdots & t(\tau_i\sigma_i^{-1}) & \cdots & t(\tau_i\sigma_i^{-1}) & \cdots & t(\tau_k\sigma_k^{-1})
\end{bmatrix},$$
Both tables $\exp(\iota(v)) = \iota(\exp(v))$ are equal as:
$$ \widetilde{t}(p_i\concat j) = \widetilde{t}(p_i)\concat \widetilde{t}(j) = \widetilde{t}(p_i)\concat \widetilde{t}(a_j),$$
$$\widetilde{t}(q_i\concat\tau_i\sigma_i^{-1}(j)) = \widetilde{t}(q_i) \concat\widetilde{t}(\tau_i\sigma_i^{-1}(j)) = \widetilde{t}(q_i) \concat t(\tau_i\sigma_i^{-1})(\widetilde{t}(j)) = \widetilde{t}(q_i) \concat t(\tau_i\sigma_i^{-1})(a_j).$$

(We have used the concat symbol ``$\concat$'' in these tables and the immediate discussion after, wherever we think it makes terms easier to read.  In our later explanations we will generally refrain from using it.)

To finish the proof, we need to show that $\iota(v) \circ \iota(u) = \iota(v \circ u)$. This follows easily from the fact that $\iota$ commutes with expansions and pushings. Without loss of generality, consider:
$$u = \begin{bmatrix}
p_1 & p_2 & \cdots & p_k \\
Id & Id & \cdots & Id \\
s_1 & s_2 & \cdots & s_k \\
\tau_1 & \tau_2 & \cdots & \tau_k
\end{bmatrix}, \quad \iota(u) = \begin{bmatrix}
\widetilde{t}(p_1) & \widetilde{t}(p_2) & \cdots & \widetilde{t}(p_k) \\
Id & Id & \cdots & Id \\
\widetilde{t}(s_1) & \widetilde{t}(s_2) & \cdots & \widetilde{t}(s_k) \\
t(\tau_1) & t(\tau_2) & \cdots &t(\tau_k)
\end{bmatrix}, $$

$$v = \begin{bmatrix}
s_1 & s_2 & \cdots & s_k \\
\tau_1 & \tau_2 & \cdots & \tau_k \\
q'_1 & q'_2 & \cdots & q'_k \\
\tau'_1 & \tau'_2 & \cdots & \tau'_k
\end{bmatrix}, \quad \iota(v) = \begin{bmatrix}
\widetilde{t}(s_1) & \widetilde{t}(s_2) & \cdots & \widetilde{t}(s_k) \\
t(\tau_1) & t(\tau_2) & \cdots &t(\tau_k)\\
\widetilde{t}(q'_1) & \widetilde{t}(q'_2) & \cdots & \widetilde{t}(q'_k) \\
t(\tau'_1) & t(\tau'_2) & \cdots &t(\tau'_k)
\end{bmatrix}. $$
Thus: 
$$\iota(v) \circ \iota(u) = \iota(v \circ u) = 
\begin{bmatrix}
\widetilde{t}(p_1) & \widetilde{t}(p_2) & \cdots & \widetilde{t}(p_k) \\
Id & Id & \cdots & Id \\
\widetilde{t}(q'_1) & \widetilde{t}(q'_2) & \cdots & \widetilde{t}(q'_k) \\
t(\tau'_1) & t(\tau'_2) & \cdots &t(\tau'_k)
\end{bmatrix}. $$
\end{proof}


\section{Algebraic Embeddings}
In this section we present some algebraic embeddings $\vmg\rightarrowtail \vnh$ between symmetric Thompson's groups (recall here $n-m=k(m-1)$ for some positive $k$, $G\leq \Sym(m)$ and with $H$ a particular extended version of $G$ in $\Sym(n)$). 

We call these embeddings ``algebraic'' as they do not arise via topological conjugacy.  In particular, these embeddings do not preserve the orbit lengths of the points of $\Cn$ (when $n>m$, which is our primary case of interest).

\subsection{Successors} Here, we give the key idea for our algebraic embeddings, which relies on extending an idea of Birget into our context.

The \textit{successor} of an element, expressed as a table, was defined in \cite{B} in order to embed $V_2(Id)$ in $V_n(Id)$, for all $n \geq 2$. We generalise Birget's definition.

\begin{definition}[Set of prefixes]\label{spref}\cite{B}
Let $P \subset \mathcal{A}_n^*$ be a prefix code of $\Cn$. We define the \textit{set of prefixes of $P$}, $\spref(P)$ as follows:
$$ \spref(P) = \{w \in \mathcal{A}_n^* : \exists p \in P,\, w <_{pref} p \}.$$
In other words, $\spref(P)$ is the set of strict prefixes of the elements of $P$.
\end{definition}

We are embedding a symmetric Thompson's group on alphabet $\mathcal{A}_m$ into a symmetric Thompson's group on alphabet $\mathcal{A}_n$, where $m\leq n$.  For this, we will assume $\mathcal{A}_m\subseteq \mathcal{A}_n$.  And in particular we set $\mathcal{A}_m=\{a_0,a_1,\ldots,a_{m-1}\}$ and $\mathcal{A}_n=\mathcal{A}_m\cup\{a_m,a_{m+1},\ldots,a_{n-1}\}$ with symbols with distinct indices being distinct (so that $|\mathcal{A}_n|=n$).

In what follows, we take a prefix code $P\subset a_{m-1}\concat \mathcal{A}_m^*$ (so each element of $P$ begins with the letter $a_{m-1}$) and transform it to a new prefix code $\succ{P}\subset \mathcal{A}_n^*$ by appending letters from the set $\{a_{m},a_{m+1},\ldots,a_n\}$.

\begin{definition}[Successor]\label{successor} Let $P \subset a_{m-1}\concat\{a_0, \dots, a_{m-1}\}^*$ be a prefix code (complete, were we to remove the initial prefix letter $a_{m-1}$, so that $|P|\equiv 1\mod (m-1)$)  with $\vert P \vert = l \geq 1$, and let $\{p_1, \dots, p_l\}$ be the ordered list of all the elements of $P$, using the \emph{reverse} dictionary order. 

We build a new prefix code $\succ{P}$ inductively using our ordered list $(p_1,p_2,\ldots,p_l)$.

Let $k$ be the smallest non-negative integer so that $n-m=k(m-1)$ (this $k$ will exist when $m$ and $n$ satisfy Higman's condition, which we require to build our embeddings).

We define (inductively) nested sets $P_{s,i}$, where $s$ will grow from $1$ to $l$, and for each value of $s$, we will have $i$ grow from $1$ to $k$.

Set $\mathcal{A}_{m,n}:= \{a_m,a_{m+1},\ldots,a_{n-1}\}$.  For every $p_s \in P$, and $i\in \{1,2,\ldots, k\}$ the \textit{$i$-th successor} $(p_s)'_i$ of $p_s$ is the element of $\spref(P)\concat \mathcal{A}_{m,n}$ defined as follows, assuming that 
\[P_{s,i-1} = \left\{\begin{matrix}(p_1)'_1,(p_1)'_{2},\ldots,(p_1)'_{k},\\(p_2)'_1,(p_2)'_{2},\ldots,(p_2)'_{k},\\
\vdots\\
(p_s)'_{1},(p_s)'_{2},\ldots,(p_s)'_{i-1}
\end{matrix}\right\}\]
has already been defined, we set:

\begin{align*}
(p_s)'_i = \min \{ & xa_j \in \spref(P)\concat\mathcal{A}_{m,n}: p_s <_{dict} xa_j  \ \mbox{and} \ xa_j \not\in P_{s,i-1}\},
\end{align*}
where $\min$ uses the dictionary order in $\{a_0, \dots, a_{n-1}\}$.
\end{definition}

\begin{example}
Suppose $m=3$ and $n=5$, so that $k=1$.  In the definition above, $a_{m-1}=2$.  So, consider the set $P=\{20,210,211,212,22\}$.  Now, $k=1$ and $\spref(P)=\{\varepsilon,2,21\}$.  We obtain
\[
\begin{array}{lll}
 p_1=22&\quad&(p_1)'_1=23\\
 p_2=212&\quad&(p_2)'_1=213\\ 
 p_3=211&\quad& (p_3)'_1=214\\
 p_4=210&\quad& (p_4)'_1=24\\
 p_5=20&\quad&(p_5)'_1=3.
\end{array}
\]

\end{example}

\begin{remark}The three constants, $n,m,k$ are not arbitrary, as the system of successors needs to be well defined.  If every element has $k$ successors, then: $$n-m = k(m-1), \, k \geq 0,$$
which is Higman's condition.\end{remark}
\begin{proof}
 If we expand an element $p_i \in P$, we need to assign successors to each element $p_ia_j$ for every $0 \leq j \leq m-1$. In particular, the number of successors $k$ of every leaf does not vary, and each element $p_ia_r$ for every $m \leq r \leq n-1$ needs to be the successor of some element in $\widetilde{P} = (P \backslash \{p_i\}) \cup \{p_ia_0, \dots,p_ia_{m-1}\}$. Then $\widetilde{P}$ has $m-1$ more elements than $P$ and there are $n-m$ new elements $p_ia_r$ for $m \leq r \leq n-1$. Thus, we need $m-1$ to evenly divide $n-m$, and $k$ is the factor of this division.
\end{proof}

We proceed to prove the following lemma, essential for the proof of Theorem \ref{thm:MT2}: 

\begin{lemma}\label{lem:successor}Suppose $m\leq n$ are naturals so that there is $k$ natural with $n-m=k(m-1)$. Suppose $l$ is a positive integer congruent to $m$ modulo $m-1$. Let $S= \{ a_m , \dots, a_{n-1}\}$ and let $P \subset a_{m-1}\concat\{a_0, \dots, a_{m-1}\}^*$ be an $l$-element prefix code, ordered as $p_l <_{dict} p_{l-1}<_{dict}\dots <_{dict} p_1$.  Let $i$ with $1\leq i\leq l$.  Then, the successors $(p_i)'_1$, $(p_i)'_2$, $\ldots$, $(p_i)'_k$ are well defined, and furthermore, the expansion in which we replace $P$ by $\widetilde{P} = (P \backslash \{p_i\} ) \cup p_i\{a_0, \dots, a_{m-1}\}$ has successors $(p_ia_j)'_i$ uniquely determined as follows: 
$$\begin{array}{lll}
(p_ia_{m-1})'_1 & = p_ia_{m}\\
& \vdots \\ (p_ia_{m-1})'_k & = p_ia_{m+k-1}\\
(p_ia_{m-2})'_1 & = p_ia_{m+k}\\
& \vdots \\ (p_ia_{m-2})'_k & = p_ia_{m+2k-1}\\
& \vdots \\
(p_ia_{1})'_1 & = p_ia_{m+(m-2)k}\\ 
& \vdots \\
(p_ia_{1})'_k & = p_ia_{m+(m-1)k-1} = p_ia_n\\
(p_ia_{0})'_1 & = (p_i)'_1\\ 
& \vdots \\
(p_ia_{0})'_k & = p_ia_{m+(m-1)k} = (p_i)'_k.\\
\end{array}$$
\end{lemma}

\begin{proof}
We prove the two statements by induction on $l$.  

{\flushleft {\it Base Case ($l=1$):}}\\
If $l=1$ then $P=\{a_{m-1}\}$.  We have $\spref{P}=\{\ew\}$.  It then follows that the $k$ successors are, the set $\{a_m,a_{m+1},\ldots,a_{m+k-1}\}$, noting that these are given in order and are the results of the inductive definition of the $k$ successors of $a_{m-1}$.  Thus we have in the base case that the successors are well defined.  We need to verify the existence of well defined successors for an expansion of $P=\{a_{m-1}\}$.  In this case, $P$ admits only one expansion, which is precisely the set $\widetilde{P}=\{a_{m-1}a_{m-1},a_{m-1}a_{m-2},\ldots,a_{m-1}a_0\}.$  We have $\spref({\widetilde{P}})=\{\ew,a_{m-1}\}$ and we have $$\begin{array}{lll}
(a_{m-1}a_{m-1})'_1 & = a_{m-1}a_{m}\\
& \vdots \\ (a_{m-1}a_{m-1})'_k & = a_{m-1}a_{m+k-1}\\
(a_{m-1}a_{m-2})'_1 & = a_{m-1}a_{m+k}\\
& \vdots \\ (a_{m-1}a_{m-2})'_k & = a_{m-1}a_{m+2k-1}\\
& \vdots \\
(a_{m-1}a_{1})'_1 & = a_{m-1}a_{m+(m-2)k}\\ 
& \vdots \\
(a_{m-1}a_{1})'_k & = a_{m-1}a_{m+(m-1)k-1} = a_{m-1}a_n\\
(a_{m-1}a_{0})'_1 & = (a_{m-1})'_1\\ 
& \vdots \\
(a_{m-1}a_{0})'_k & = a_{m-1}a_{m+(m-1)k} = (a_{m-1})'_k.\\
\end{array}$$
We can directly observe these successors are well defined and distinct.  Thus, the statement is true for $l=1$.

{\flushleft {\it Inductive Case ($l>1$):}}\\
Now let us assume that $\widetilde{P}$ is a result of $v$ expansions from the one-element prefix code $\{a_{m-1}\}$, for some $v\geq 1$, where for any prefix code resulting from $u$ expansions from $\{a_{m-1}\}$ for $0\leq u<v$ the statement of the lemma holds.  We will show that the successors of our expansion $\widetilde{P}$ are well defined.
In particular, if $P$ is the prefix code (of size $l$) arising from doing only the first $t-1$ expansions from $\{a_{m-1}\}$ towards the prefix code $\widetilde{P}$, and $p_i$ is the element of $P$ which is being replaced by and $m$-fold expansion to create $\widetilde{P}$, then by induction, the successors $(p_i)'_1$, $(p_i)'_2$, $\ldots$, $(p_i)'_k$ of $p_i\in P$ are well defined.

Note that $P_{i-1,k} = \widetilde{P}_{i-1,k}$, as the involved subsets of both $P$ and $\widetilde{P}$ are equal. Thus the first successor to assign is $(p_ia_{m-1})'_1$. Suppose that $(p_ia_{m-1})'_1 <_{dict} p_ia_m$, then $(p_ia_{m-1})'_1 = pa_t$ for some $p \in \spref(P): p <_{dict} p_i$ and $a_t \in \{a_m, \dots,a_{n-1} \}$. Thus, before expanding $P$, $pa_t$ is one of the successors of some $p_r \in P$. 

If $p_r <_{dict} p_i$, then the set of successors of $p_i$ is defined before the set of successors of $p_r$. Because $p_i$ is expanded, the set of $k$ successors of $p_ia_{m-1}$ is equal to the set of successors of $p_i$, and this set does not contain $pa_t$, which is a contradiction. On the other hand, if $p_r >_{dict} p_i$, then $pa_t \in P_{i-1,k} = \widetilde{P}_{i-1,k}$, which is also a contradiction. Thus $(p_ia_{m-1})'_1 = p_ia_m$. We can use a similar argument for all $(p_ia_{m-1})'_1  \dots (p_ia_{1})'_k$. 

For $p_ia_0$, all successors of the form $p_ia_s,\, a_s \in \{a_m ,\dots, a_{n-1}\}$, have already been assigned.  Thus, the remaining $k$ successors are precisely the $k$ successors of $p_i$, taken in order.
\end{proof}

\begin{remark}\label{rem:successor}
 Birget in \cite{B} gives a formula for the $i$-th successor of an element, for the case of $m=2$.  The statement of Lemma \ref{lem:successor} above shows the natural generalisation of that formula holds when we have the Higman Condition (as we must for successors to be well defined).  The resulting formula is given as follows:
 
Let $P \subset a_{m-1}\concat\{a_0, \dots, a_{m-1}\}^*$ be a prefix code with $\vert P \vert \geq 2$, such that the elements of $P$ are ordered in reverse dictionary order. Then every element of $w \in P$ can be written uniquely in the form $ua_ia_0^t$, where $u\in \{a_0 , \dots, a_{m-1}\}^*$ and $t \geq 0$. The $i$-th successor of $w$ is:
$$(w)'_i = (ua_ja_0^t)'_i = ua_{m-1+(m-1-j)k + i}.$$ 
We stress that this formula is only valid if $P$ is ordered in reverse dictionary order.
\end{remark} 

\subsection{The algebraic embedding}
We proceed to define the algebraic embedding of $\vmg$ in $\vnh = V_n(G_{ext})$. Let $g \in \vmg$, given by the following table:
\begin{align*}
g & = \begin{bmatrix}
p_1 & p_2 & \cdots & p_l \\
\sigma_1 & \sigma_2 & \cdots & \sigma_l\\
q_1 & q_2 & \cdots & q_l \\
\tau_1 & \tau_2 & \cdots & \tau_l
\end{bmatrix}
\end{align*}
We define the embedding $\iota(g)$ below.  The resulting tables are large, and our notation requires some explanation.  The idea of the embedding is to use the identity map initially, and at $a_{m+k}$ and later letters, but under the address $a_{m-1}$ we place the prefix code $p_1$ to $p_l$, and we also require action under the successors.  The first row then has entries following the ordered list given here (wrapped at natural locations due to page length constraints):
\[
\begin{matrix}
a_0,a_1,\ldots,a_{m-2},\\
a_{m-1}p_1,a_{m-1}p_2,\ldots,a_{m-1}p_l,\\
(a_{m-1}p_1)'_1,(a_{m-1}p_2)'_1,  \ldots, (a_{m-1}p_l)'_1,\\
(a_{m-1}p_1)'_2,(a_{m-1}p_2)'_2,  \ldots, (a_{m-1}p_l)'_2,\\
\dots\\
(a_{m-1}p_1)'_k,(a_{m-1}p_2)'_k,  \ldots, (a_{m-1}p_l)'_k,\\
a_{m+k},a_{m+k+1},\ldots,a_{n-1}.
\end{matrix}
\]
We use vertical bars ``$\vert$'' in our table at the same locations that we placed line-wraps in the row detailed above, for clarity of grouping.  The element $\iota(g)$ is now given by the following table:
\begin{flushleft}
$\begin{array}{r}
\left[\rule{0cm}{1cm} \setlength\arraycolsep{1pt} \begin{array}{ccccccccccccc}
a_0 & \cdots & a_{m-2} & \vert & a_{m-1}p_1 & \cdots & a_{m-1}p_l & \vert & (a_{m-1}p_1)'_1 & \cdots &(a_{m-1}p_l)'_1 &\vert & \cdots \\
Id & \cdots & Id & \vert & \sigma'_1 &\cdots & \sigma'_l & \vert & \sigma'_1 &\cdots & \sigma'_l & \vert & \cdots\\
a_0 & \cdots & a_{m-2} & \vert & a_{m-1}q_1 & \cdots & a_{m-1}q_l & \vert & (a_{m-1}q_1)'_1 & \cdots &(a_{m-1}q_l)'_1 &\vert & \cdots \\
Id & \cdots & Id & \vert & \tau'_1 &\cdots & \tau'_l & \vert & \tau'_1 &\cdots & \tau'_l & \vert & \cdots  \\
\end{array} \color{white}\right] \\ \\

 \color{white} =  \left[\rule{0cm}{1cm} \setlength\arraycolsep{1pt} \color{black} \begin{array}{ccccccccccccc}
\cdots  & \vert & (a_{m-1}p_1)'_k & \cdots & (a_{m-1}p_l)'_k & \vert & a_{m+k} & \cdots & a_{n-1} \\
\cdots & \vert &  \sigma'_1 &\cdots & \sigma'_l & \vert & Id & \cdots & Id \\
\cdots  & \vert & (a_{m-1}q_1)'_k & \cdots & (a_{m-1}q_l)'_k & \vert & a_{m+k} & \cdots & a_{n-1} \\
\cdots & \vert & \tau'_1 &\cdots & \tau'_l & \vert & Id & \cdots & Id \\ 
\end{array} \right]
\end{array}$
\end{flushleft}

Note that the set of successors of $P = \{p_1, \dots, p_l\}$ are assigned supposing that $p_l <_{dict} \dots <_{dict} p_1$. Therefore, the set of successors of $Q = \{q_1, \dots, q_l\}$ is assigned following the order $q_l \rightarrow \dots \rightarrow q_1$, which does not need to follow the dictionary order on $Q$. 

Indeed, the first and third rows of $\iota(g)$ are both prefix codes of $\Cn$. On the one hand, suppose that the number of columns of $g$ is $l = m + d(m-1)$ for some $d \geq 0$. It follows that the number of columns of $\iota(g)$ whose elements of the first row start with $a_{m-1}$ is $n+ d(n-1)$ (observe that the last $k$ terms from the successor substitution will not begin with $a_{m-1}$). On the other hand, as the number of columns of $g$ is $(m + d(m-1))$, and we assign $k$ successors to every column, we have $(m+d(m-1))(k+1)$ columns on $\iota(g)$. As we have Higman's Condition, the reader can verify that $(m+d(m-1))(k+1) = n+ d(n-1) +k$.  From this, we see firstly that $(n-1)\vert k$, but more importantly, this embedding/successor operation does not place any constraints on the  number of expansions $d$ that were used to create the original prefix code for the domain of $g$.

\begin{proof}[Proof of Theorem \ref{thm:MT2}]
If we push down the action of every $\sigma_i$, we have:
\begin{align*}
\push(g) & = \begin{bmatrix}
p_1 & p_2 & \cdots & p_l \\
Id & Id & \cdots & Id\\
q_1 & q_2 & \cdots & q_l \\
\tau_1\sigma_1^{-1}  & \tau_2\sigma_2^{-1} & \cdots & \tau_l\sigma_l^{-1} \\
\end{bmatrix}  
\end{align*}
Thus the table for $\iota(\push(g))$ is:
\begin{flushleft}
$\begin{array}{r}
\left[\rule{0cm}{1cm} \setlength\arraycolsep{1pt} \begin{array}{ccccccccccccc}
a_0 & \cdots & a_{m-2} & \vert & a_{m-1}p_1 & \cdots & a_{m-1}p_l & \vert & (a_{m-1}p_1)'_1 & \cdots &(a_{m-1}p_l)'_1 &\vert & \cdots \\
Id & \cdots & Id & \vert & Id &\cdots & Id & \vert & Id &\cdots & Id & \vert & \cdots\\
a_0 & \cdots & a_{m-2} & \vert & a_{m-1}q_1 & \cdots & a_{m-1}q_l & \vert & (a_{m-1}q_1)'_1 & \cdots &(a_{m-1}q_l)'_1 &\vert & \cdots \\[2pt]
Id & \cdots & Id & \vert & (\tau_1\sigma_1^{-1})' &\cdots & (\tau_l\sigma_l^{-1})' & \vert & (\tau_1\sigma_1^{-1})' &\cdots & (\tau_l\sigma_l^{-1})' & \vert & \cdots  \\
\end{array} \color{white}\right] \\ \\

 \color{white} =  \left[\rule{0cm}{1cm} \setlength\arraycolsep{1pt} \color{black} \begin{array}{ccccccccccccc}
\cdots  & \vert & (a_{m-1}p_1)'_k & \cdots & (a_{m-1}p_l)'_k & \vert & a_{m+k} & \cdots & a_{n-1} \\
\cdots & \vert &  Id &\cdots & Id & \vert & Id & \cdots & Id \\
\cdots  & \vert & (a_{m-1}q_1)'_k & \cdots & (a_{m-1}q_l)'_k & \vert & a_{m+k} & \cdots & a_{n-1} \\[2pt]
\cdots & \vert & (\tau_1\sigma_1^{-1})' &\cdots & (\tau_l\sigma_l^{-1})' & \vert & Id & \cdots & Id \\
\end{array} \right]
\end{array}$
\end{flushleft}
On the other hand the table for $\push(\iota(g))$ is:
\begin{flushleft}
$\begin{array}{r}
\left[\rule{0cm}{1cm} \setlength\arraycolsep{1pt} \begin{array}{ccccccccccccc}
a_0 & \cdots & a_{m-2} & \vert & a_{m-1}p_1 & \cdots & a_{m-1}p_l & \vert & (a_{m-1}p_1)'_1 & \cdots &(a_{m-1}p_l)'_1 &\vert & \cdots \\
Id & \cdots & Id & \vert & Id &\cdots & Id & \vert & Id &\cdots & Id & \vert & \cdots\\
a_0 & \cdots & a_{m-2} & \vert & a_{m-1}q_1 & \cdots & a_{m-1}q_l & \vert & (a_{m-1}q_1)'_1 & \cdots &(a_{m-1}q_l)'_1 &\vert & \cdots \\[2pt]
Id & \cdots & Id & \vert & \tau'_1(\sigma'_1)^{-1} &\cdots & \tau'_l(\sigma'_l)^{-1} & \vert & \tau'_1(\sigma'_1)^{-1} &\cdots & \tau'_l(\sigma'_l)^{-1} & \vert & \cdots  \\
\end{array} \color{white}\right] \\ \\

  \color{white} =  \left[\rule{0cm}{1cm} \setlength\arraycolsep{1pt} \color{black} \begin{array}{ccccccccccccc}
\cdots  & \vert & (a_{m-1}p_1)'_k & \cdots & (a_{m-1}p_l)'_k & \vert & a_{m+k} & \cdots & a_{n-1} \\
\cdots & \vert &  Id &\cdots & Id & \vert & Id & \cdots & Id \\
\cdots  & \vert & (a_{m-1}q_1)'_k & \cdots & (a_{m-1}q_l)'_k & \vert & a_{m+k} & \cdots & a_{n-1} \\[2pt]
\cdots & \vert & \tau'_1(\sigma'_1)^{-1} &\cdots & \tau'_l(\sigma'_l)^{-1} & \vert & Id & \cdots & Id \\ 
\end{array} \right]
\end{array}$
\end{flushleft}

Recall from the statement of Theorem \ref{thm:MT2} that for an element $\tau\in\Sym(m)$, the extended version $\tau'$ of $\tau$ in $\Sym(n)$ is that element of $\Sym(n)$ which agrees with $\tau$ on the set $\mathcal{A}_m$ and acts as the identity on the points of $\mathcal{A}_{m,n}$ in $\mathcal{A}_n$.  Thus, both tables are equal, as $(\tau_i\sigma_i^{-1})' = (\tau'_i)(\sigma'_i)^{-1},\, \forall i \in \{i,\dots,l\}$. If we expand $g$ on $p_i$:

$$\exp(g) = \begin{bmatrix}
p_1 & p_2 & \cdots & p_ia_0 & \cdots & p_ia_{m-1} & \cdots  &  p_l \\
Id & Id & \cdots & Id & \cdots & Id& \cdots & Id\\
q_1 & q_2 & \cdots & q_i \tau_i(a_0) & \cdots & q_i\tau_i(a_{m-1}) & \cdots  &  q_l \\
\tau_1 & \tau_2 & \cdots & \tau_i & \cdots & \tau_i& \cdots & \tau_l
\end{bmatrix}$$
Then the table for $\iota(\exp(g))$ is:
\begin{flushleft}
$\begin{array}{r}
\left[\rule{0cm}{1cm} \setlength\arraycolsep{1pt} \begin{array}{ccccccccccccc}
a_0 & \cdots & a_{m-2} & \vert & a_{m-1}p_1 & \cdots & a_{m-1}p_ia_0 & \cdots & a_{m-1}p_ia_{m-1} & \cdots & a_{m-1}p_l &\vert & \cdots \\
Id & \cdots & Id & \vert & Id &\cdots & Id & \cdots & Id &\cdots  & Id & \vert & \cdots\\
a_0 & \cdots & a_{m-2} & \vert & a_{m-1}q_1 & \cdots & a_{m-1}q_i\tau_i(a_0) & \cdots & a_{m-1}q_i\tau_i(a_{m-1}) & \cdots & a_{m-1}q_l &\vert & \cdots \\
Id & \cdots & Id & \vert & \tau'_1 &\cdots & \tau'_i  &\cdots & \tau'_i  &\cdots & \tau'_l & \vert & \cdots  \\
\end{array} \color{white}\right] \\ \\
\color{white} =  \left[\rule{0cm}{1cm} \setlength\arraycolsep{1pt} \color{black} \begin{array}{ccccccccccccc}
\cdots  & \vert  & \cdots  & (a_{m-1}p_ia_0)'_j & \cdots & (a_{m-1}p_ia_{m-1})'_j & \cdots & \vert & a_{m+k} & \cdots & a_{n-1} \\
\cdots  & \vert  & \cdots  & Id & \cdots & Id & \cdots & \vert & Id &\cdots & Id \\
\cdots   & \vert  & \cdots  & (a_{m-1}q_i\tau_i(a_0))'_j & \cdots & (a_{m-1}q_i\tau_i(a_{m-1}))'_j & \cdots& \vert & a_{m+k} & \cdots & a_{n-1}  \\
\cdots  & \vert& \cdots & \tau'_i &\cdots & \tau'_i & \cdots & \vert & Id & \cdots & Id \\ 
\end{array} \right]
\end{array}$
\end{flushleft}

On the other hand, the table for $\exp(\iota(g))$ is:

\begin{flushleft}
$\begin{array}{r}
\left[\rule{0cm}{1cm} \setlength\arraycolsep{1pt} \begin{array}{ccccccccccccc}
a_0 & \cdots & a_{m-2} & \vert & a_{m-1}p_1 & \cdots & a_{m-1}p_ia_0 & \cdots & a_{m-1}p_ia_{n-1} & \cdots & a_{m-1}p_l &\vert & \cdots \\
Id & \cdots & Id & \vert & Id &\cdots & Id & \cdots & Id &\cdots  & Id & \vert & \cdots\\
a_0 & \cdots & a_{m-2} & \vert & a_{m-1}q_1 & \cdots & a_{m-1}q_i\tau'_i(a_0) & \cdots & a_{m-1}q_i\tau'_i(a_{n-1}) & \cdots & a_{m-1}q_l &\vert & \cdots \\
Id & \cdots & Id & \vert & \tau'_1 &\cdots & \tau'_i  &\cdots & \tau'_i  &\cdots & \tau'_l & \vert & \cdots  \\
\end{array} \color{white}\right] \\ \\
\color{white} =  \left[\rule{0cm}{1cm} \setlength\arraycolsep{1pt} \color{black} \begin{array}{ccccccccccccc}
\cdots  & \vert &\cdots & (a_{m-1}p_1)'_j & \cdots & (a_{m-1}p_i)'_j & \cdots & (a_{m-1}p_l)'_j & \cdots & \vert & a_{m+k} & \cdots & a_{n-1} \\
\cdots & \vert &\cdots &  Id &\cdots &  Id &\cdots & Id &\cdots & \vert & Id & \cdots & Id \\
\cdots  & \vert & \cdots & (a_{m-1}q_1)'_j  & \cdots & (a_{m-1}q_i)'_j & \cdots & (a_{m-1}q_l)'_j & \cdots & \vert & a_{m+k} & \cdots & a_{n-1} \\
\cdots & \vert &\cdots & \tau'_1 &\cdots & \tau'_i &\cdots & \tau'_l & \cdots & \vert & Id & \cdots & Id \\ 
\end{array} \right]
\end{array}$
\end{flushleft}
It is straightforward to check that for both tables the columns starting at $a_{m-1}p_ia_s$ for $0 \leq s \leq m-1$ are equal, as $\tau'_i(s) = \tau_i(s), \forall s \in  \{0,\dots, m-1\}$.

For $ m \leq s \leq n-1$ by Lemma \ref{lem:successor}, we know there is a correspondence between the first two rows of $\iota(\exp(g))$ and the first two rows of $\exp(\iota(g))$. On the other hand, from the description of the algebraic embedding, the order in which successors for every $q_j$ are selected depends only on the order of $\{p_1, \dots, p_l\}$, so we have the following calculations:
$$\begin{array}{lllll}
(a_{m-1}q_i\tau_i(a_{m-1}))'_1 & = a_{m-1}q_ia_{m} & = a_{m-1}q_i\tau'_i(a_{m}) \\
& \vdots & \vdots &\\ 
(a_{m-1}q_i\tau_i(a_{m-1}))'_k & = a_{m-1}q_ia_{m+k-1} & = a_{m-1}q_i\tau'_i(a_{m+k-1}) \\
(a_{m-1}q_i\tau_i(a_{m-2}))'_1 & = a_{m-1}q_ia_{m+k} & = a_{m-1}q_i\tau'_i(a_{m+k}) \\
& \vdots &\vdots & \\
(a_{m-1}q_i\tau_i(a_{m-2}))'_k  & = a_{m-1}q_ia_{m+2k-1} & = a_{m-1}q_i\tau'_i(a_{m+2k-1})\\
& \vdots& \vdots & \\
(a_{m-1}q_i\tau_i(a_{1}))'_1  & = a_{m-1}q_ia_{m+(m-2)k} & = a_{m-1}q_i\tau'_i(a_{m+(m-2)k})\\ 
& \vdots &\vdots & \\
(a_{m-1}q_i\tau_i(a_{1}))'_k & = a_{m-1}q_ia_n & = a_{m-1}q_i\tau'_i(a_{m+(m-1)k-1}) \\
(a_{m-1}q_i\tau_i(a_{0}))'_1 & = (a_{m-1}q_i)'_1\\ 
& \vdots \\
(a_{m-1}q_i\tau_i(a_{0}))'_k & =  (a_{m-1}q_i)'_k.\\
\end{array}$$
That is, e.g., $(a_{m-1}q_i\tau_i(a_{m-1}))'_1$ comes first in the choice of successor as $(a_{m-1}p_ia_{m-1})'_1$ appears first under the order of the $p_i$, independent of $\tau_i$.  Thus, the latter two rows of these tables are also equivalent.

Finally, it is easy to see that $\iota(h \circ g) = \iota(h) \circ \iota(g)$, as $\iota$ commutes with expansions and pushings. We only need to obtain row equality on the first part of the table as the remaining part depends entirely on $P$ (resp. on $Q$ for the element $h$): 
\begin{align*}
\iota(g) & = \begin{bmatrix}
\cdots  & a_{m-1}p_i & \cdots & (a_{m-1}p_i)'_j & \cdots \\
\cdots  & Id & \cdots & Id &\cdots  \\
\cdots  & a_{m-1}q_i& \cdots & (a_{m-1}q_i)'_j & \cdots\\
\cdots &\tau'_i & \cdots & \tau'_i&\cdots\\ 
\end{bmatrix},\\ 
\end{align*}
\begin{align*}
\iota(h) & = \begin{bmatrix}
 \cdots & a_{m-1}q_i& \cdots & (a_{m-1}q_i)'_j& \cdots\\
\cdots & \tau'_i  & \cdots& \tau'_i & \cdots\\
 \cdots & a_{m-1}r_i & \cdots & (a_{m-1}r_i)'_j& \cdots\\
 \cdots & \tau''_i & \cdots& \tau''_i& \cdots\\ 
\end{bmatrix},\\ 
\iota(h) \circ \iota(g) & = \begin{bmatrix}
 \cdots & a_{m-1}p_i& \cdots & (a_{m-1}p_i)'_j& \cdots\\
\cdots & Id  & \cdots& Id & \cdots\\
 \cdots & a_{m-1}r_i & \cdots & (a_{m-1}r_i)'_j& \cdots\\
 \cdots & \tau''_i & \cdots& \tau''_i& \cdots \\ 
\end{bmatrix} = \iota(h \circ g).
\end{align*}  Thus the result follows. 
\end{proof}
\label{Bibliography}
\bibliographystyle{abbrv} 
\bibliography{Bibliography} 

\end{document}